\documentclass[english,british]{extarticle}
\usepackage[T1]{fontenc}
\usepackage[latin9]{inputenc}
\usepackage{verbatim}
\usepackage{amsmath}
\usepackage{amsthm}
\usepackage{amssymb}
\usepackage{esint}

\makeatletter
\theoremstyle{plain}
\newtheorem{thm}{\protect\theoremname}
  \theoremstyle{plain}
  \newtheorem{lem}[thm]{\protect\lemmaname}
  \theoremstyle{plain}
  \newtheorem{cor}[thm]{\protect\corollaryname}

\usepackage{babel}

\author{Mic Bowman\\The Pennsylvania State University \and Saumya K.~Debray and Larry L.~Peterson\\The University of Arizona}

\usepackage{babel}
\addto\captionsbritish{\renewcommand{\corollaryname}{Corollary}}
  \addto\captionsbritish{\renewcommand{\lemmaname}{Lemma}}
  \addto\captionsbritish{\renewcommand{\theoremname}{Theorem}}
  \addto\captionsenglish{\renewcommand{\corollaryname}{Corollary}}
  \addto\captionsenglish{\renewcommand{\lemmaname}{Lemma}}
  \addto\captionsenglish{\renewcommand{\theoremname}{Theorem}}
  \providecommand{\corollaryname}{Corollary}
  \providecommand{\lemmaname}{Lemma}
\providecommand{\theoremname}{Theorem}

\makeatother

\usepackage{babel}
  \addto\captionsbritish{\renewcommand{\corollaryname}{Corollary}}
  \addto\captionsbritish{\renewcommand{\lemmaname}{Lemma}}
  \addto\captionsbritish{\renewcommand{\theoremname}{Theorem}}
  \addto\captionsenglish{\renewcommand{\corollaryname}{Corollary}}
  \addto\captionsenglish{\renewcommand{\lemmaname}{Lemma}}
  \addto\captionsenglish{\renewcommand{\theoremname}{Theorem}}
  \providecommand{\corollaryname}{Corollary}
  \providecommand{\lemmaname}{Lemma}
\providecommand{\theoremname}{Theorem}

\begin{document}

\title{Steady state for the subcritical contact branching random walk on
the lattice with the arbitrary number of offspring and with immigration\foreignlanguage{english}{}\thanks{\selectlanguage{english}%
The study has been funded by the Russian Science Foundation (project
No.~17-11-01098)\selectlanguage{british}%
}\foreignlanguage{english}{ }}

\author{{\normalsize{}{}Elena Chernousova }{\small{}{}}\\
 {\small{}{}Department of Mathematical Basics of Control}\\
 {\small{}{}\vspace{0.3cm}
Moscow Institute of Physics and Technology, Dolgoprudny, Russian Federation}\\
 {\normalsize{}{}Yaqin Feng}{\small{}{}}\\
 {\small{}{}Department of mathematics }\\
 {\small{}{}\vspace{0.3cm}
Ohio University, Athens, Ohio 45701,USA}\\
 {\normalsize{}{}Stanislav Molchanov }{\small{}{}}\\
 {\small{}{}Department of Mathematics and Statistics }\\
 {\small{}{}University of North Carolina at Charlotte, Charlotte,
NC 28223,USA, }\\
 {\small{}{}\vspace{0.3cm}
National Research University, Higher School of Economics, Russian
Federation}\\
 {\normalsize{}{}Joseph Whitmeyer}{\small{}{}}\\
 {\small{}{}Department of Sociology}\\
 {\small{}{} University of North Carolina at Charlotte, Charlotte,
NC 28223,USA}}
\maketitle
\begin{abstract}
We consider the subcritical contact branching random walk on $\mathbb{Z}^{d}$
in continuous time with the arbitrary number of offspring and with
immigration. We prove the existence of the steady state (statistical
equilibrium). 
\end{abstract}

\section{Introduction}

This paper is the continuation of our previous publication \cite{CE}.
As in the majority of the publications in the area of population dynamics,
we considered binary splitting in \cite{CE}. During time interval
$[t,t+dt]$, each particle in our population either dies with probability
$\mu dt$ or produces with probability $\beta dt$ one offspring which
jumps from the parental particle at the size $x\mathbb{\in Z}^{d}$
to the random point $x+z$ with probability $b(z)$. We assume $b(z)=b(-z)$
and $\sum_{z\in\mathbb{Z}^{d}}b(z)=1.$ In other terms, the infinitesimal
generating function describing the branching has a form $\phi(z)=\beta z^{2}-(\beta+\mu)z+\mu$.
\cite{SM_JW_2016,SM_JW_2017,Koralov_2013,KPP,key-4} are also based
on the binary splitting.

But in many applications, especially in the model of forest introduced
in \cite{key-4}, where the particles(i.e.trees) do not move at all
but produce the seeds which are randomly distributed around the parental
tree (we introduce this option in our more general model (\ref{backward Kolmogorov  equations})),
the assumption of the binary splitting is highly artificial. In contrast,
the natural assumption here is that typical number of seeds is large
(of order hundreds and thousands). i.e. The infinitesimal generating
function has now a form $\phi(z)=\sum_{l=2}^{\infty}\beta_{l}\psi_{z}^{l}-(\sum_{l=2}^{\infty}\beta_{l}+\mu)\psi_{z}+\mu$.
It is well known that for heavy tailed distribution $\beta_{l},$
the branching process can explode. Since we will use low moment analysis
and the Carleman type conditions for the uniqueness of the solution
of the moments problem, we will assume that$\{\beta_{l},l\ge2\}$
have geometrically decay, i.e. $\phi(z)$ is analytic in the circle
$|z|\leq1+\delta,\delta>0.$

Let us note that for the arbitrary number of offspring, corresponding
moments ( mean numbers of offspring, variance of this number etc.)
can be essentially different. It leads to important phenomenon of
the intermittency in the model of the forest.

In this paper, we study the steady state for the subcritical branching
random walk on the lattice with the arbitrary number of offspring.
It is not only the natural coninutaion of the publication \cite{CE},
which consider the binary splitting. It is also a natural continuation
of the pubilication \cite{CE_SM}, which study the convergence of
the population to the statistical equilibrium for critical contact
process on the lattice $\mathbb{Z}^{d}$. We consider the general
model of the subcritical branching random walk on the lattice $\mathbb{Z}^{d}$.
The structure of this paper is as follows:\foreignlanguage{english}{
In section 2, we introduce our model, containing the random walk with
generator $\mathcal{L}_{a}$, mortality rate $\mu$, splitting rate
with arbitrary number of offspring and their distribution around parental
particle with some law and immigration rate $\gamma$. We provide
in this section several technical lemmas. In section 3, we prove a
Carleman type estimate for the cumulants of subcritical populations
and prove the existence of a steady state. The last section contains
the summary.}

\section{Description of the model}

Let $N(t,y)$ to be the particles field on the lattice $\mathbb{Z}^{d}$
with continuous time $t\geq0$, i.e. $N(t,y)$ is the numbers of population
at site $y\in\mathbb{Z}^{d}$ at the moment $t$. The evolution of
this particle field consists of several elements: 
\begin{itemize}
\item Each particle independently on others performs (until the transformation:
death or splitting) the random walk $X(t)$ with the generator $\mathcal{L}_{a}$,
where 
\begin{equation}
\kappa\mathit{\mathcal{L}_{a}}f(x)=\kappa\sum_{z\in\mathbb{Z}^{d}\setminus\{0\}}(f(x+z)-f(z))a(z).\label{generator}
\end{equation}
We assume that 
\begin{eqnarray*}
a(z)=a(-z)\ge0,\sum_{z\neq0,z\in\mathbb{Z}^{d}}a(z)=1
\end{eqnarray*}
and the random walk $X(t)$ is supported on the full lattice (but
not on some sub-lattice). It means that $\forall y\in\mathbb{Z}^{d},$
there is some integer $k\ge1$, there are some vectors $z_{1},\cdots,z_{k}$
and there are some positive integers $n_{1},\cdots,n_{k}$ such that
$y=\sum_{i=1}^{k}n_{i}z_{i}$ and $a(z_{i})>0$ for $i=1,\ldots,k$.
\item Each particle in the site $x$ during the time interval $[t,t+dt]$
(independently on others and past time) can annihilate (die) with
probability $\mu\,dt$ or splits onto $l$ particles with probabilities
$\beta_{l}\,dt$ where $l\ge2$. In such splitting, one offspring
(it can be considered the parental particle) remains at $x$ and the
other $l-1$~particles jump independently from $x$ to $x+v$ with
probability distribution $b(v)$, where $b(v)=b(-v)$ and $\sum_{v\in\mathbb{Z}^{d}\setminus\{0\}}b(v)=1$.
We assume that 
\begin{equation}
\Delta=\mu-\sum_{l=2}^{\infty}(l-1)\beta_{l}>0.\label{Delta}
\end{equation}
\item We also assume that for any site $x$, the new particles (immigrants)
appear at the moments $0<\tau_{1}(x)<\tau_{2}(x)\cdots$ and $\tau_{i+1}(x)-\tau_{i}(x)\sim Exp(\gamma)$.
In different terms, moments $\tau_{i}(x),$ $i\geq1$ form a Poissonian
point field on $\{x\}\times[0,\infty)$ with parameter $\gamma$.
Meanwhile, we assume the independence of such point fields for different
$x\in\mathbb{Z}^{d}.$\\

\selectlanguage{english}%
Let $n(t-\tau_{i}(x),x,y)$ denote the subpopulation, i.e., the number
of particles, at site $y\in\mathbb{Z}^{d}$ at time $t$ descended
from a particle that appeared at $x$ (immigrated) at time $\tau_{i}(x)<t$.
\foreignlanguage{british}{Without loss of generality, we can assume
that $N(0,y)\equiv0$, since all subpopulation starting at the moment
$t=0$ will vanish to the large moment $t$ with probability at least
$e^{-\Delta t}$. As a result, we have the following important representation
\begin{eqnarray}
N(t,y)=\sum_{x\in\mathbb{Z}^{d}}\sum_{\tau_{i}(x)\leq t}n(t-\tau_{i}(x),x,y),\label{representation of N(t,y)}
\end{eqnarray}
where subpopulation $n(t-\tau_{i}(x),x,y)$ are independent for different
$x\in\mathbb{Z}^{d}$ and $\tau_{i}\leq t$. 
\begin{eqnarray*}
N(t,y) & \underset{=}{\text{Law}} & \sum_{x\in\mathbb{Z}^{d}}\sum_{\tau_{i}(x)\leq t}n(t-\tau_{i}(x),x,y),\\
 & \underset{=}{\text{Law}} & \sum_{x\in\mathbb{Z}^{d}}\sum_{\xi_{1}+\cdots+\xi_{k}\leq t}n(\xi_{1}+\cdots+\xi_{k},x,y),
\end{eqnarray*}
where $\xi_{i}\sim Exp(\gamma)$. }
\end{itemize}
\selectlanguage{british}%
Let us consider the subpopulation $n(t,x,y)$. We introduce the generating
function for an individual subpopulation 
\begin{equation}
u_{z}(t,x;y)=\mathbb{E}z^{n(t,x,y)}.\label{eq:defintion}
\end{equation}
We hereafter consider this as a function of the variables $t$ and
$x$. For every fixed $y\in\mathbb{Z}^{d}$, $u_{z}(t,x;y)$ satisfies
the backward Kolmogorov equation (where we omit the arguments $(t,x;y)$):
\begin{equation}
\frac{\partial u_{z}}{\partial t}=\kappa\mathcal{L}_{a}u_{z}-\left(\mu+\sum_{l=2}^{\infty}\beta_{l}\right)u_{z}+\mu+u_{z}\sum_{l=2}^{\infty}\beta_{l}\left(u_{z}*b\right)^{l-1}
\label{backward Kolmogorov  equations}
\end{equation}
with initial condition $u_{z}(0,x;y)=z$ if $x=y$ and $u_{z}(0,x;y)=1$
otherwise. Here, we use the following designation for the convolution
of two functions 
\begin{equation}
u_{z}*b=(u_{z}*b)(t,x;y)=\sum_{v\in\mathbb{Z}^{d}}u_{z}(t,x-v;y)b(v).\label{convolution}
\end{equation}

From ~\eqref{backward Kolmogorov  equations} we can derive equations
for the factorial moments 
\begin{equation}
m_{k}(t,x;y)=\mathbb{E}\left[n(t,x,y)(n(t,x,y)-1)\cdots(n(t,x,y)-k+1)\right]=\frac{\partial^{k}u_{z}}{\partial z^{k}}\biggr|_{z=1}(t,x;y),\label{def k factorial moment}
\end{equation}
where $k=1,2,\ldots$ In particular, by differentiating Eq.~\eqref{backward Kolmogorov  equations}
we obtain an equation for the first moment: 
\begin{equation}
\begin{split} & \frac{\partial m_{1}}{\partial t}=\left(\kappa\mathcal{L}_{a}+\sum_{l=2}^{\infty}(l-1)\beta_{l}\mathcal{L}_{b}\right)m_{1}+\Delta m_{1},\\
 & m_{1}(0,x;y)=\delta(x-y).
\end{split}
\label{eq m_1}
\end{equation}
Here, $\mathcal{L}_{b}$ is defined as (similarly to Eq.~\eqref{generator}):
\begin{equation}
\mathcal{L}_{b}f=\left(\mathcal{L}_{b}f\right)(x)=\sum_{v\neq0}b(v)\left[f(x+v)-f(x)\right].
\end{equation}
The solution of \eqref{eq m_1} is: 
\begin{equation}
m_{1}(t,x,y)=e^{-\Delta t}p(t,x,y),\label{m_1}
\end{equation}
where $p(t,x,y)$ (fundamental solution) is the transition probability
of the event that a particle that starts at~$x\in\mathbb{Z}^{d}$
arrives at~$y\in\mathbb{Z}^{d}$ during time $t>0$ for the random
walk which is defined by the symmetric isotropic generator $\kappa\mathcal{L}_{a}+\sum_{l=2}^{\infty}(l-1)\beta_{l}\mathcal{L}_{b}$.
i.e. $p(t,x,y)$ satisfies the following equation

\begin{equation}
\begin{split} & \frac{\partial p(t,x,y)}{\partial t}=\left(\kappa\mathcal{L}_{a}+\sum_{l=2}^{\infty}(l-1)\beta_{l}\mathcal{L}_{b}\right)p(t,x,y)\\
 & p(0,x,y)=\delta_{x}(y).
\end{split}
\label{p(t,x,y)-1}
\end{equation}

\selectlanguage{english}%
Denote $\hat{p}(t,k,0)=\sum_{x}e^{ikx}p(t,x,0)$. Applying Fourier
transform on both side of the Eq. \ref{p(t,x,y)-1}, we have

\begin{eqnarray*}
\frac{\partial\hat{p}(t,k,0)}{\partial t} & =\left(\kappa\mathcal{\hat{L}}_{a}(k)+\sum_{l=2}^{\infty}(l-1)\beta_{l}\mathcal{\hat{L}}_{b}(k)\right)\hat{p}(t,k,0)
\end{eqnarray*}
where $\mathcal{\hat{L}}_{a}(k)=1-\hat{a}(k)$ , $\mathcal{\hat{L}}_{b}(k)=1-\hat{b}(k)$,
$\hat{a}(k)=\sum_{z}cos(k,z)a(z)$, $\hat{b}(k)=\sum_{z}cos(k,z)b(z)$.
As a result, 
\begin{eqnarray*}
\hat{p}(t,k,0) & = & e^{t\left(\kappa\mathcal{\hat{L}}_{a}(k)+\sum_{l=2}^{\infty}(l-1)\beta_{l}\mathcal{\hat{L}}_{b}(k)\right)}.
\end{eqnarray*}
Therefore, the transition probability of the underlying random walk
has the form
\begin{eqnarray*}
p(t,x,y) & =\frac{1}{\left(2\pi\right)^{d}}\int_{T^{d}} & e^{t\left(\kappa\mathcal{\hat{L}}_{a}(k)+\sum_{l=2}^{\infty}(l-1)\beta_{l}\mathcal{\hat{L}}_{b}(k)\right)}e^{-ik(x-y)}dk,
\end{eqnarray*}
 and $T^{d}=[-\pi,\pi]^{d}.$

\selectlanguage{british}%
Note that 

\begin{equation}
\sum_{y\in\mathbb{Z}^{d}}p(t,x,y)=1,\label{p(t,x,y)-2-1}
\end{equation}

\begin{equation}
p(t,x,y)\leq p(t,x,x)=p(t,0,0)\label{p(t,x,y)-2}
\end{equation}

For each $x\in\mathbb{Z}^{d}$, $\nu_{x}(t)=\sum_{y\in\mathbb{Z}^{d}}n(t,x,y)$
is a~Galton-Watson process, see~\cite{Sevast}. We have the well
known equation for the generating function of this process $\psi_{z}(t):=Ez^{\nu_{x}(t)}:$
\[
\begin{split} & \frac{\partial\psi_{z}}{\partial t}=\sum_{l=2}^{\infty}\beta_{l}\psi_{z}^{l}-(\sum_{l=2}^{\infty}\beta_{l}+\mu)\psi_{z}+\mu=\\
 & (\psi_{z}-1)\Biggl(\sum_{l=2}^{\infty}\beta_{l}(\psi_{z}^{l-1}+\psi_{z}^{l-2}+\cdots+\psi_{z})-\mu\Biggr),\\
 & \psi_{z}(0)=z.
\end{split}
\]

Please refer to \cite{Sevast} for more details of discussion for
$\psi_{z}(t)$.

\section{Main result}

The central goal of this paper is to prove the convergence of the
particle field $N(t,y)$, $y\in\mathbb{Z}^{d}$ to a steady state
(statistical equilibrium).
\begin{thm}
\label{main result} Let $N(t,y)$, $y\in\mathbb{Z}^{d}$ be a random
field as described above, assume that  for all $l\ge2$, 

\[
\beta_{l}\le\beta\delta^{l}\,\,for\,\,some\,\,\beta>0,\delta\in(0,1).
\]
Then, for all $y\in\mathbb{Z}^{d}$ 
\begin{equation}
N(t,y)\xrightarrow{\text{Law}}N(\infty,y).
\end{equation}
\noindent\textbf{Remark:}
\end{thm}

\begin{itemize}
\item The condition that $\beta_{l}$ decreases geometrically implies that
the generating function of this sequence $\sum_{l=2}^{\infty}\beta_{l}z^{l}$
is analytic in the disk $|z|<\delta$ for some suitable $\delta>0$.
\item In order to prove Theorem \eqref{main result}, we first will estimate
all factorial moments of a subpopulation, i.e. $m_{k}(t,x;y)$, $k\ge1$,
see Eq.~\eqref{def k factorial moment}. From this and the relationship
between moments and cumulants, we can estimate cumulants for the total
population $N(t,y)$ uniformly in $t$. Using the monotonicity in
$t$ and boundedness of these cumulants, we can conclude that their
limit exists at $t\to\infty$. Then we will use the Carleman conditions
to establish a unique limiting distribution.
\end{itemize}
Under the model assumption, It is trivial that $m_{0}(t,x;y)\equiv u_{z}(t,x;y)\mid_{z=1}=1$.

For all $k\ge2$, differentiate Eq.~\eqref{backward Kolmogorov  equations}
$k$-times differentiation, we can derive equation for the $k$-th
factorial moments: 
\begin{equation}
\begin{split} & \frac{\partial m_{k}}{\partial t}=\left(\kappa\mathcal{L}_{a}+\sum_{l=2}^{\infty}(l-1)\beta_{l}\mathcal{L}_{b}\right)m_{k}-\Delta m_{k}+\\
 & \sum_{l=2}^{\infty}\beta_{l}\sum_{n=1}^{k-1}\frac{m_{n}}{n!}\sum_{{{\sum\nolimits _{s=1}^{l-1}j_{s}=k-n,}\atop {j_{s}\ge0}}}\frac{k!}{j_{1}!\cdots j_{l}!}\left(m_{j_{1}}*b\right)\cdot\ldots\cdot\left(m_{j_{l-1}}*b\right)+\\
 & \sum_{l=2}^{\infty}\beta_{l}\sum_{{{\sum\nolimits _{s=1}^{l-1}j_{s}=k,}\atop {0\le j_{s}\le k-1}}}\frac{k!}{j_{1}!\cdots j_{l}!}\left(m_{j_{1}}*b\right)\cdot\ldots\cdot\left(m_{j_{l-1}}*b\right)
\end{split}
\label{eq m_k}
\end{equation}
with the initial condition $m_{k}(0,x;y)=0$ when $k\ge2$. Without
loss of generality, we assume that $y=0$. 

Let us first recall Duhamel\textquoteright s principal. 
\begin{lem}
(Duhamel\textquoteright s principal) if $f(t,x)$, $t\ge0$, $x\in\mathbb{Z}^{d}$
is the fundamental solution of the homogeneous equation: 
\begin{equation}
\frac{\partial f}{\partial t}(t,x)=\mathcal{L}f(t,x)
\end{equation}
with the initial condition $f(0,x)=\delta(x)$, then the equation:
\begin{equation}
\frac{\partial F}{\partial t}(t,x)=\mathcal{L}F(t,x)+h(t,x)
\end{equation}
with the initial condition $F(0,x)=0$ has the solution: 
\begin{equation}
F(t,x)=\int_{0}^{t}ds\sum_{z\in\mathbb{Z}^{d}}f(t-s,x-z)h(s,z).
\end{equation}
\end{lem}

In order to the estimate all factorial moments of a subpopulation,
the proof of the next lemma will be similar to the proofs in \cite{CE_SM}. 
\begin{lem}
\label{upper_bound_for_factor_moments} Under the conditions of Theorem~\ref{main result},
for all $k\ge1$ 
\begin{equation}
m_{k}(t,x;0)\le k!B^{k-1}D_{k}e^{-\Delta t}p(t,x,0),\label{formula_upper_bound_for_factor_moments}
\end{equation}
where 
\begin{equation}
B=\max\left\{ 1,\beta\int\limits _{0}^{\infty}e^{-\Delta s}p(s,0,0)ds\right\} <\infty\label{def_B}
\end{equation}
and the sequence $D_{k}$ is recursively defined as: $D_{1}=1$ and,
for $k\ge2$ 
\begin{equation}
\begin{split}D_{k}= & \sum_{l=2}^{\infty}\delta^{l}\sum_{n=1}^{k-1}D_{n}\sum_{i=1}^{l-1}{l-1 \choose i}\sum_{{{\sum\nolimits _{s=1}^{i}j_{s}=k-n,}\atop {j_{s}\ge1}}}D_{j_{1}}\cdot\ldots\cdot D_{j_{i}}+\\
 & \sum_{l=2}^{\infty}\delta^{l}\sum_{i=2}^{l-1}{l-1 \choose i}\sum_{{{\sum\nolimits _{s=1}^{i}j_{s}=k,}\atop {j_{s}\ge1}}}D_{j_{1}}\cdot\ldots\cdot D_{j_{i}}.
\end{split}
\label{D_k}
\end{equation}
\end{lem}

\textit{Proof:} Denote $\tilde{m}_{k}(t,x;0)=\frac{m_{k}(t,x;0)}{k!}$,
$M_{j}=\tilde{m}_{k}*b=\sum_{v\in\mathbb{Z}^{d}}b(v)\tilde{m}_{j}(t,x+v;0)$,
and $\mathcal{L}_{a,b}=\kappa\mathcal{L}_{a}+\sum_{l=2}^{\infty}(l-1)\beta_{l}\mathcal{L}_{b}$.
Then, Eq.~\eqref{eq m_k} has the form 
\begin{equation}
\begin{split}\frac{\partial\tilde{m}_{k}}{\partial t}=\mathcal{L}_{a,b}\tilde{m}_{k}+ & \sum_{l=2}^{\infty}\beta_{l}\sum_{n=1}^{k-1}\tilde{m}_{n}\sum_{{{\sum\nolimits _{s=1}^{l-1}j_{s}=k-n,}\atop {j_{s}\ge0}}}M_{j_{1}}\cdot\ldots\cdot M_{j_{l-1}}+\\
 & \sum_{l=2}^{\infty}\beta_{l}\sum_{{{\sum\nolimits _{s=1}^{l-1}j_{s}=k,}\atop {0\le j_{s}\le k-1}}}M_{j_{1}}\cdot\ldots\cdot M_{j_{l-1}}.
\end{split}
\end{equation}
From Duhamel's formula, we obtain that 
\begin{equation}
\begin{split} & \tilde{m}_{k}(t,x;0)=\\
 & \int_{0}^{t}dse^{-\Delta(t-s)}\sum_{z\in\mathbb{Z}^{d}}p(t-s,x-z,0)\sum_{l=2}^{\infty}\beta_{l}\sum_{n=1}^{k-1}\tilde{m}_{n}\sum_{{{\sum\nolimits _{s=1}^{l-1}j_{s}=k-n,}\atop {j_{s}\ge0}}}M_{j_{1}}\cdot\ldots\cdot M_{j_{l-1}}(s,z;0)+\\
 & \int_{0}^{t}dse^{-\Delta(t-s)}\sum_{z\in\mathbb{Z}^{d}}p(t-s,x-z,0)\sum_{l=2}^{\infty}\beta_{l}\sum_{{{\sum\nolimits _{s=1}^{l-1}j_{s}=k,}\atop {0\le j_{s}\le k-1}}}M_{j_{1}}\cdot\ldots\cdot M_{j_{l-1}}(s,z;0).
\end{split}
\label{proof 1}
\end{equation}
 If we excluding $M_{0}\equiv1$, then the inner sum of the first
term in Eq.~\eqref{proof 1} can be formatted as: 
\begin{equation}
\begin{split} & \sum_{n=1}^{k-1}\tilde{m}_{n}\sum_{{{\sum\nolimits _{s=1}^{l-1}j_{s}=k-n,}\atop {j_{s}\ge0}}}M_{j_{1}}\cdot\ldots\cdot M_{j_{l-1}}=\\
 & \sum_{n=1}^{k-1}\tilde{m}_{n}\sum_{i=1}^{l-1}{l-1 \choose i}\sum_{{{\sum\nolimits _{s=1}^{i}j_{s}=k-n,}\atop {j_{s}\ge1}}}M_{j_{1}}\cdot\ldots\cdot M_{j_{i}}.
\end{split}
\label{eq:proof_add}
\end{equation}
and the inner sum of the second term in Eq.~\eqref{proof 1} can
be written as: 
\begin{equation}
\sum_{{{\sum\nolimits _{s=1}^{l-1}j_{s}=k,}\atop {0\le j_{s}\le k-1}}}M_{j_{1}}\cdot\ldots\cdot M_{j_{l-1}}=\sum_{i=2}^{l-1}{l-1 \choose i}\sum_{{{\sum\nolimits _{s=1}^{i}j_{s}=k,}\atop {j_{s}\ge1}}}M_{j_{1}}\cdot\ldots\cdot M_{j_{i}}.\label{proof 2a}
\end{equation}
In the following, we will prove the lemma using mathematical induction
.

For $k=1$, 
\[
\tilde{m}_{1}(t,x;0)=p(t,x,0).
\]

and\foreignlanguage{english}{ $p(t,x,0)$ is} the fundamental solution
of Eq.~\eqref{m_1} and the base of induction is verified.

Let's assume that Eq.~\eqref{formula_upper_bound_for_factor_moments}
is true for $k-1$. Then, the right-hand side of Eq.~\eqref{eq:proof_add}
is bounded by
\begin{equation}
\begin{split} & B^{k-1}e^{-\Delta s}p(s,z;0)\sum_{n=1}^{k-1}D_{n}\sum_{i=1}^{l-1}{l-1 \choose i}\left(\frac{e^{-\Delta s}(p*b)}{B}\right)^{i}\sum_{{{\sum\nolimits _{s=1}^{i}j_{s}=k-n,}\atop {j_{s}\ge1}}}D_{j_{1}}\cdot\ldots\cdot D_{j_{i}}\le\\
 & B^{k-1}e^{-\Delta s}p(s,z;0)\frac{p(s,0,0)e^{-\Delta s}}{B}\sum_{n=1}^{k-1}D_{n}\sum_{i=1}^{l-1}{l-1 \choose i}\sum_{{{\sum\nolimits _{s=1}^{i}j_{s}=k-n,}\atop {j_{s}\ge1}}}D_{j_{1}}\cdot\ldots\cdot D_{j_{i}},
\end{split}
\label{proof 3}
\end{equation}
where we use simple facts that for all $x,y\in\mathbb{Z}^{d}$ $p(t,x,y)\le p(t,0,0)$
from Eq.\ref{p(t,x,y)-2} and $(p*b)(t,x,0)\le p(t,x,0)$. 

Indeed, 

\begin{eqnarray*}
(p*b)(t,x,0) & = & \sum_{z}b(z)p(t,x-z,0)\\
 & = & \frac{1}{(2\pi)^{d}}\int_{T^{d}}\hat{p}(t,k,0)\hat{b}(k)e^{-ikx}\,dk\\
 & \leq & \frac{1}{(2\pi)^{d}}\int_{T^{d}}\hat{p}(t,k,0)e^{-ikx}\,dk\\
 & = & p(t,x,0).
\end{eqnarray*}
Here we use the fact that $\hat{p}(t,k,0)$ and $\hat{b}(k)$ are
real and not larger than $1$.

From the definition of $B$, see Eq.~\eqref{def_B}, for all $i\ge1$
we have 
\[
\left(\frac{(p*b)(s,z,0)e^{-\Delta s}}{B}\right)^{i}\le\frac{e^{-\Delta s}p(s,z,0)}{B}\le\frac{e^{-\Delta s}p(s,0,0)}{B}.
\]

Analogously, the right-hand side of Eq.~\eqref{proof 2a} is bounded
by
\begin{equation}
\begin{split} & B^{k-1}p(s,z,0)e^{-\Delta s}\sum_{i=2}^{l-1}{l-1 \choose i}\left(\frac{e^{-\Delta s}p*b}{B}\right)^{i-1}\sum_{{{\sum\nolimits _{s=1}^{i}j_{s}=k,}\atop {j_{s}\ge1}}}D_{j_{1}}\cdot\ldots\cdot D_{j_{i}}\le\\
 & B^{k-1}p(s,z,0)e^{-\Delta s}\frac{e^{-\Delta s}p(s,0,0)}{B}\sum_{i=2}^{l-1}{l-1 \choose i}\sum_{{{\sum\nolimits _{s=1}^{i}j_{s}=k,}\atop {j_{s}\ge1}}}D_{j_{1}}\cdot\ldots\cdot D_{j_{i}}.
\end{split}
\label{proof 3a}
\end{equation}

Now we can substitute it into Eq.~\eqref{proof 1}: 
\begin{equation}
\begin{split}\tilde{m}_{k}(t,x;0)\le & B^{k-1}e^{-\Delta t}\int_{0}^{t}ds\frac{e^{-\Delta s}p(s,0,0)}{B}\sum_{z\in\mathbb{Z}^{d}}p(t-s,x-z,0)p(s,z,0)\cdot\\
 & \sum_{l=2}^{\infty}\beta_{l}\biggl(\sum_{n=1}^{k-1}D_{n}\sum_{i=1}^{l-1}{l-1 \choose i}\sum_{{{\sum\nolimits _{s=1}^{i}j_{s}=k-n,}\atop {j_{s}\ge1}}}D_{j_{1}}\cdot\ldots\cdot D_{j_{i}}+\\
 & \sum_{i=2}^{l-1}{l-1 \choose i}\sum_{{{\sum\nolimits _{s=1}^{i}j_{s}=k,}\atop {j_{s}\ge1}}}D_{j_{1}}\cdot\ldots\cdot D_{j_{i}}\biggr)
\end{split}
\end{equation}
Base on the following facts: 
\begin{itemize}
\item $\sum_{z\in\mathbb{Z}^{d}}p(t-s,x-z,0)p(s,z,0)=\sum_{z\in\mathbb{Z}^{d}}p(t-s,x,z)p(s,z,0)=p(t,x,0)$
from Chapman-Kolmogorov equation;
\item $\beta_{l}\le\beta\delta^{l}$ from assumption of the lemma;
\item $\frac{\beta\int_{0}^{t}e^{-\Delta s}p(s,0,0)ds}{B}\le1$ from Eq.~\eqref{def_B},
\end{itemize}
We can state the lemma using the recursive definition of the sequence
$D_{k}$ Eq.~\eqref{D_k} $\Box$

\begin{lem}
\label{geometrically growth D_k} The sequence $D_{k}$ that is determined
by $D_{1}=1$ and Eq.~\eqref{D_k} increases no faster than geometrically.

\end{lem}

The geometrically growth of $D_{k}$ states in Lemma \eqref{geometrically growth D_k}
is proved in Lemma 2 in \cite{CE_SM}. From Lemma \eqref{upper_bound_for_factor_moments}
and Lemma \eqref{geometrically growth D_k}, we have the following
Corollary. 
\begin{cor}
\label{m_k(t,x;0)} 
\begin{equation}
m_{k}(t,x;0)\le c^{k}k!e^{-\Delta t}p(t,x,0)
\end{equation}
for all $k\ge1$ and 
\begin{equation}
\sum_{x\in\mathbb{Z}^{d}}m_{k}(t,x;0)\le c^{k}k!e^{-\Delta t}.
\end{equation}
\end{cor}

Let us now introduce the notation for cumulants. For any random variable
$X$, let $\phi_{X}(z)=Ez^{X}$, then the $l$-th~cumulant 
\begin{eqnarray*}
\chi_{l}(X)=\frac{d^{l}}{dz^{l}}\ln(\phi_{X}(z))\mid_{z=1}.
\end{eqnarray*}
In general, the relationship between moments and cumulants is given
by 
\begin{eqnarray}
\chi_{l}(X)=l!\sum\frac{(-1)^{j_{1}+\cdots+j_{l}-1}(j_{1}+\cdots+j_{l}-1)}{j_{1}!\cdot\ldots\cdot j_{l}!}\prod_{k=1}^{l}\left(\frac{m_{k}(X)}{k!}\right)^{j_{k}}\label{eq:cumulant and moments}
\end{eqnarray}
and 
\begin{eqnarray}
m_{l}(X)=l!\sum\frac{1}{j_{1}!\cdot\ldots\cdot j_{l}!}\prod_{k=1}^{l}\left(\frac{\chi_{k}(X)}{k!}\right)^{j_{k}}\label{eq:moments and cumulant}
\end{eqnarray}
where the sign $\sum$ means the sum over all non-negative integers
($j_{1},\cdots,j_{l})$ satisfying the constraint 
\begin{eqnarray*}
1j_{1}+2j_{2}+3j_{3}+\cdots+lj_{l}=l.
\end{eqnarray*}

One important property of cumulants is additivity: for independent
random variables $X$ and $Y$, $\chi_{l}(X+Y)=\chi_{l}(X)+\chi_{l}(Y)$.

Due to previous remark we obtain that 
\begin{equation}
\begin{split}\chi_{l}(N(t,0)) & =\chi_{l}\Biggl(\sum_{x\in\mathbb{Z}^{d}}\sum_{\tau_{i}(x)\le t}n(t-\tau_{i}(x),x,0)\Biggr)\\
 & =\sum_{x\in\mathbb{Z}^{d}}\chi_{l}\Biggl(\sum_{\tau_{i}(x)\le t}n(t-\tau_{i}(x),x,0)\Biggr).
\end{split}
\label{chi_l(N)}
\end{equation}
In order to calculate $\chi_{l}\Biggl(\sum_{\tau_{i}(x)\le t}n(t-\tau_{i}(x),x,0)\Biggr)$
, we will prove the following Lemma. 

\begin{lem}
Let $\xi$ be a random variable uniformly distributed on $[0,t]$,
then 
\[
\chi_{l}\Biggl(\sum_{\tau_{i}(x)\le t}n(t-\tau_{i}(x),x,0)\Biggr)=(\gamma t)m_{l}\Biggl(n(\xi,x,0)\Biggr).
\]
\begin{proof}
The generating function of $\chi_{l}\Biggl(\sum_{\tau_{i}(x)\le t}n(t-\tau_{i}(x),x,0)\Biggr)$
has the simple form: 
\[
\begin{split}F(z) & =Ez^{\sum_{\tau_{i}(x)\le t}n(t-\tau_{i}(x),x,0)}\\
 & =Ez^{\sum_{i=1}^{\Pi_{x}(t)}n(\xi,x,0)}\\
 & =\sum_{k=0}^{\infty}e^{-\gamma t}\frac{(\gamma t)^{l}}{l!}\Biggl(Ez^{n(\xi,x,0)}\Biggr)^{l}=exp\Biggl\{-\gamma t(1-Ez^{n(\xi,x,0)})\Biggr\},
\end{split}
\]
where $\Pi_{x}(t)$ is a Poissonian process with parameter $\gamma$
and we use the fact that, if $\Pi_{x}(t)=l$, then the moments of
this process has the distribution of the ordered statistics of $l$
uniformly distributed random variables on $[0,t]$ .

The log-generating function is 
\begin{equation}
\begin{split}\ln F(z) & =-\gamma t(1-Ez^{n(\xi,x,0)})\\
 & =-\gamma t\Biggl(1-\sum_{l=0}^{\infty}\frac{m_{l}(n(\xi,x,0))}{l!}(z-1)^{l}\Biggr)\\
 & =\gamma t\sum_{l=1}^{\infty}\frac{m_{l}(n(\xi,x,0))}{l!}(z-1)^{l}.
\end{split}
\label{11}
\end{equation}
At the same time, 
\begin{equation}
\ln F(z)=\sum_{l=1}^{\infty}\frac{\chi_{l}\Biggl(\sum_{\tau_{i}(x)\le t}n(t-\tau_{i}(x),x,0)\Biggr)}{l!}(z-1)^{l}.\label{12}
\end{equation}
From \eqref{11} and \eqref{12} we obtain the statement of the lemma. 
\end{proof}
\begin{cor}
\label{chi_l(N)} $\chi_{l}(N(t,0))$ is a monotone function of time
$t$ and 
\[
\chi_{l}(N(t,0))=\gamma\int_{0}^{t}\sum_{x\in\mathbb{Z}^{d}}m_{l}(s,x,0)\,ds.
\]
\end{cor}

\end{lem}

From Corollary \ref{m_k(t,x;0)} and Corollary \ref{chi_l(N)} we
obtain 

\begin{cor}
\[
\chi_{l}(N(t,0))\le c^{l}l!\gamma\int_{0}^{t}e^{-\Delta s}p(s,x,0)\,ds\le c^{l}l!\frac{\gamma}{\Delta}.
\]
\end{cor}

The last gives an upper bound uniformly in $t$ for the cumulants
of total population $N(t,0)$. Using this and the monotonicity in
$t$ of the cumulants of the total population $\chi_{l}\left(N(t,\cdot)\right)$
,we conclude the existence and boundedness of $\chi_{l}\left(N(\infty,\cdot)\right)$:
\begin{equation}
\chi_{l}\left(N(\infty,\cdot)\right)\le C^{l}l!.\label{bound chi_k N(infty,y)}
\end{equation}

Finally we may conclude that the behaviour in the limit of the cumulants
of the total population $\chi_{l}\left(N(\infty,y)\right)$ determines
uniquely the limit distribution of $N(\infty,y)$, $y\in\mathbb{Z}^{d}$.
In other words, the classic problem of moments \cite{Feller} does
not take place in this situation. The upper boundary in Eq.~\eqref{bound chi_k N(infty,y)}
implies that the log-generating function for $N(\infty,\cdot)$ is
analytical in some neighbourhood of $z=1$, which is why the sequence
of $\chi_{l}\left(N(\infty,\cdot)\right)$ uniquely determines the
probability distribution of $N(\infty,\cdot)$ \cite[Chapter~VII, S~6]{Feller}.
Traditionally, these conditions on the sequence of moments or cumulant
that are sufficient for the existence of a uniquely determined distribution
law are called the Carleman conditions. 

\noindent\textbf{\textit{Remark:}}
\begin{itemize}
\item Similar to the discussion in our previous work \cite{CE}, one can
perform similar analysis in the case when $0<\Delta^{-}\le\Delta(x)=\mu(x)-\sum_{l=2}^{\infty}(l-1)\beta_{l}(x)\le\Delta^{+}<\infty.$
The proof of boundedness of cumulants and moments will be similar
and we can prove a result analogous to Theorem \ref{main result}
and there is a limiting distribution in this case as well. 
\end{itemize}

\section{Conclusion}

We considered a subcritical contact branching random walk on the lattice
with the arbitrary number of offspring and with immigration. We showed
that, if the rate of mortality is larger than the average number of
new particles per unit time (subcritical case), and the tail of the
distribution of the number of offspring decreases at least geometrically,
then the probability distribution of the population converges to a
limiting distribution.

\end{document}